\documentclass[12pt,reqno]{amsart}
\usepackage{amssymb,a4wide,eucal,amscd,yhmath,amsgen,amsopn}
\usepackage{fancyhdr}
\usepackage{amsmath,hyperref,xcolor}
\usepackage{mathrsfs}
\usepackage[all,cmtip]{xy}
\SilentMatrices
\SelectTips{eu}{}
\newtheorem{theorem}{Theorem}[section]
\newtheorem{proposition}[theorem]{Proposition}

\newtheorem{lemma}[theorem]{Lemma}

\theoremstyle{rem}

\theoremstyle{definition}
\newtheorem{definition}[theorem]{Definition}
\theoremstyle{construct}

\theoremstyle{examp}

\newcommand\projective\mathbf
\newcommand\PP{\projective P}

\newcommand\OO{\mathcal O}

\newcommand\ZZ{\mathbb Z}

\newcommand\GG{\mathbb Gr}

\newcommand\onto\twoheadrightarrow
\newcommand\lra\longrightarrow
\newcommand\dar\downarrow
\DeclareMathOperator{\pic}{Pic}
\DeclareMathOperator{\im}{im}
\DeclareMathOperator{\cok}{coker}
\DeclareMathOperator{\rk}{rank}
\DeclareMathOperator{\Hom}{Hom}

\begin{document}

\title{Monads on multiprojective spaces and associated vector bundles}
\author{Damian M Maingi}
\date{April, 2022}
\keywords{Monads, multiprojective spaces, simple vector bundles}

\address{Department of Mathematics\\Sultan Qaboos University\\ P.O Box 50, 123 Muscat, Oman\\
Department of Mathematics\\University of Nairobi\\P.O Box 30197, 00100 Nairobi, Kenya\\https://orcid.org/0000-0001-9267-9388}
\email{dmaingi@squ.edu.om, dmaingi@uonbi.ac.ke} 

\maketitle

\begin{abstract}
In this paper we establish the existence of monads on Cartesian products of projective spaces. 
We construct vector bundles associated to monads  on 
$\PP^{a_1}\times\PP^{a_1}\times\PP^{a_2}\times\PP^{a_2}\times\cdots\times\PP^{a_n}\times\PP^{a_n}$.
Once the  monad on $X$ exists the next natural question is if the cohomology vector bundle associated to these monads are simple or not. 
We study these vector bundles associated to monads on $X$ and prove their stability and simplicity.
\end{abstract}

\section{Introduction}
\noindent The goal of this paper is the construction of indecomposable vector bundles on algebraic varieties motivated by
the famous Hartshorne's conjecture concerning the non-existence of indecomposable rank 2 vector bundles on 
$n$-dimensional projective spaces for $n\geq7$\cite{3}. Hartshorne\cite{4} gave a list of problems, the first two of which 
concern finding vector bundles of low rank for large dimension projective spaces. 

In the last forty years there are few examples of indecomposable vector bundles of low rank $r$, say for $r\leq n-2$ on $\PP^n$, $n\geq4$.
Some examples of low rank indecomposable vector bundles over projective varieties are;
in positive characteristic $p\neq2$ Kumar \cite{11} constructed many indecomposable rank 2 bundles on $\PP^4$. 
Kumar, Peterson and Rao A P  constructed of low rank vector bundles on $\PP^4$ and $\PP^5$. 
In characteristic zero, there is the famous Horrocks-Mumford bundle of rank 2 over $\PP^4$\cite{8}, there is also the Horrocks vector bundle of rank 3 on $\PP^5$\cite{8}. 
Tango constructed the Tango bundles\cite{18} of rank $n-1$ on $\PP^n$ for $n\geq3$ and the rank 2 vector bundle on $P^5$ in characteristic 2 
by Tango\cite{19} and these were obtained as cohomologies of certain monads.\\

\noindent One of the most important tools or technique to construct these vector bundles is via monads which 
appear in many contexts within algebraic geometry. Monads were first introduced by Horrocks\cite{7} who showed  
that all vector bundles $E$ on $\PP^3$ could be obtained as the cohomology bundle of a monad of the following kind:
\[
\xymatrix
{
0\ar[r] & \oplus_i\OO_{\PP^3}(a_i) \ar[r]^{A} & \oplus_j\OO_{\PP^3}(b_j) \ar[r]^{B} & \oplus_n\OO_{\PP^3}(c_n) \ar[r] & 0
}
\]
where $A$ and $B$ are matrices whose entries are homogeneous polynomials of degrees $b_j-a_i$ and $c_n-b_j$ respectively for some integers $i,j,n$.\\

\noindent The first problem is to establish the existence of monads on the given algebraic variety. 
Fl\o{}ystad\cite{2} gave sufficient and necessary conditions for the existence of monads over the projective space. 
Costa and Miro-Roig \cite{1} extended these results to smooth quadric hypersurfaces of dimension at least 3.
Marchesi, Marques and Soares\cite{13} generalized Fl\o{}ystad's theorem to a larger set of varieties. Maingi\cite{13} established the
existence of monads on $\PP^n\times\PP^n$ and proved simplicity of the cohomology bundle. 
Most recently, Maingi\cite{15} extended the results to  $\PP^n\times\PP^n\times\PP^m\times\PP^m$ .\\
\\
In this paper we establish the existence of monads on multiprojective spaces $\PP^{a_1}\times\PP^{a_1}\times\PP^{a_2}\times\PP^{a_2}\times\cdots\times\PP^{a_n}\times\PP^{a_n}$. We first extend Fl\o{}ystad's\cite{2} main theorem to 
\\
We first establish the existence of monads
\[\begin{CD}0@>>>{\OO_X(-1,\cdots,-1)^{\oplus k}} @>>^{f}>{\mathscr{G}_n\oplus\cdots\oplus\mathscr{G}_m}@>>^{g}>\OO_X(1,\cdots,1)^{\oplus k} @>>>0\end{CD}\]
on $X=\PP^{a_1}\times\PP^{a_1}\times\PP^{a_2}\times\PP^{a_2}\times\cdots\times\PP^{a_n}\times\PP^{a_n}$
where 
\begin{align*}
\mathscr{G}_1:=\OO_X(-1,0,0,\cdots,0)^{\oplus a_1+\oplus k}\oplus\OO_X(0,-1,0,0,\cdots,0)^{\oplus a_1+\oplus k}\\
\mathscr{G}_2:=\OO_X(0,0,-1,\cdots,0)^{\oplus a_2+\oplus k}\oplus\OO_X(0,0,0,-1,\cdots,0)^{\oplus a_2+\oplus k}\\
\cdots\cdots\cdots\cdots\cdots\cdots\cdots\cdots\cdots\cdots\cdots\cdots\cdots\cdots\cdots\cdots\cdots\cdots\cdots\\
\mathscr{G}_n:=\OO_X(0,0,\cdots,0,-1,0)^{\oplus a_n+\oplus k}\oplus\OO_X(0,0,\cdots,0,-1)^{\oplus a_n+\oplus k}
\end{align*}
We then prove stability of the kernel bundle $\ker g$ and finally prove that the cohomology vector bundle, $E=\ker g/\im f$ is simple.\\
\\
\noindent In the following section we outline and set up the key concepts that we shall use to prove the main theorems. 
These definitions and notations can be found in Chapter 2 of \cite{15} by Okonek, Schneider and Spindler.

\section{Preliminaries}

\begin{definition}
Let $X$ be a nonsingular projective variety. 
\begin{enumerate}
\renewcommand{\theenumi}{\alph{enumi}}
 \item A {\it{monad}} on $X$ is a complex of vector bundles:
\[
\xymatrix{0\ar[r] & M_1 \ar[r]^{\alpha} & M_0 \ar[r]^{\beta} & M_2 \ar[r] & 0}
\]
with $\alpha$ injective and $\beta$ surjective equivalently, $M_\bullet$ is a monad if $\alpha$ and $\beta$ are of maximal rank and $\beta\circ\alpha = 0$.
\item A monad as defined above has a display diagram of short exact sequences as shown below:
\[
\begin{CD}
@.@.0@.0\\
@.@.@VVV@VVV\\
0@>>>{M_1} @>>>\ker{\beta}@>>>E@>>>0\\
@.||@.@VVV@VVV\\
0@>>>{M_1} @>>^{\alpha}>{M_0}@>>>\cok{\alpha}@>>>0\\
@.@.@V^{\beta}VV@VVV\\
@.@.{M_2}@={M_2}\\
@.@.@VVV@VVV\\
@.@.0@.0
\end{CD}
\]
\item The kernel of the map $\beta$, $\ker\beta$ and the cokernel of $\alpha$, $\cok\alpha$ for the given monad are also vector bundles and the vector bundle
$E = \ker(\beta)/\im (\alpha)$ and is called the cohomology bundle of the monad.
\item The rank of $E$ is given by, $\rk(E)= \rk(M_0)-\rk(M_1)-\rk(M_2)$.
\item The $i^{th}$ chern class of $E$ is given by, $c_i(E)= c_i(M_0)c_i(M_1)^{-1}c_i(M_2)^{-1}$.
\end{enumerate}
\end{definition}

\begin{definition}
Let $X$ be a nonsingular projective variety, let $\mathscr{L}$ be a very ample line sheaf, and $V,W,U$ be finite dimensional $k$-vector spaces.
A linear monad on $X$ is a complex of sheaves,
\[ M_\bullet:
\xymatrix
{
0\ar[r] & V\otimes {\mathscr{L}}^{-1} \ar[r]^{A} & W\otimes \OO_X \ar[r]^{B} & U\otimes \mathscr{L} \ar[r] & 0
}
\]
where $A\in \Hom(V,W)\otimes H^0 \mathscr{L}$ is injective and $B\in \Hom(W,U)\otimes H^0 \mathscr{L}$ is surjective.\\
The existence of the monad $M_\bullet$ is equivalent to the following conditions on $A$ and $B$
\begin{enumerate}
\renewcommand{\theenumi}{\alph{enumi}}
 \item $A$ and $B$ are of maximal rank.
 \item $BA$ is the zero matrix.
\end{enumerate}

\end{definition}

\begin{definition}
A torsion-free sheaf $E$ on $X$ is said to be a {\it{linear sheaf}}  on $X$ if it can be represented as the
cohomology sheaf of a linear monad i.e.
$E= \ker(\beta)/\im (\alpha)$, moreover $\rk(E) = w - u - v$, where $w=\dim W$, $v=\dim V$ and $u=\dim U$.
\end{definition}

\begin{definition}
Let $X$ be a non-singular irreducible projective variety of dimension $d$ and let $\mathscr{L}$ be an ample line bundle on $X$. For a 
torsion-free sheaf $F$ on $X$ we define
\begin{enumerate}
\renewcommand{\theenumi}{\alph{enumi}}
 \item the degree of $F$ relative to $\mathscr{L}$ as $\deg_{\mathscr{L}}F:= c_1(F)\cdot \mathscr{L}^{d-1}$, where $c_1(F)$ is the first chern class of $F$
 \item the slope of $F$ as $\mu_{\mathscr{L}}(F):= \frac{c_1(F)\mathscr{L}^{d-1}}{rk(F)}$.
 \end{enumerate}
\end{definition}

\begin{definition}
Let $X$ be an algebraic variety and let $E$ be a torsion-free sheaf on $X$. Then $E$ is $\mathscr{L}$-stable
if every subsheaf $F\hookrightarrow E$ satisfies $\mu_{\mathscr{L}}(F)<\mu_{\mathscr{L}}(E)$, where $\mathscr{L}$ is an ample invertible sheaf.
\end{definition}

\subsection{Hoppe's Criterion over cyclic varieties.}
Suppose that the picard group Pic$(X) \simeq \ZZ$ such varieties are called cyclic.
Given a locally free sheaf (or, equivalently, a holomorphic vector bundle) $E\rightarrow X$, 
there is a unique integer $k_E$ such that $-r + 1 \leq c_1(E(-k_E)) \leq 0$.
Setting $E_{norm} := E(-k_E)$, we say $E$ is normalized if $E = E_{norm}$. 
Then one has the following stability criterion:

\begin{proposition} [\cite{6}, Lemma 2.6]
Let $E$ be a rank $r$ holomorphic vector bundle over a cyclic projective variety $X$. 
\begin{enumerate}
\renewcommand{\theenumi}{\alph{enumi}}
 \item If $H^0((\bigwedge^q E)_{norm}) = 0$  for  $1\leq q\leq r-1$, then E is stable.
 \item If $H^0((\bigwedge^q E)_{norm}(-1)) = 0$ for $1\leq q \leq r-1$, then E is semistable.
\end{enumerate}
\end{proposition}

\subsection{Hoppe's Criterion over polycyclic varieties.}
Suppose that the picard group Pic$(X) \simeq \ZZ^l$ where $l\geq2$ is an integer then $X$ is a polycyclic variety.
Given a divisor $B$ on $X$ we define $\delta_{\mathscr{L}}(B):= \deg_{\mathscr{L}}\OO_{X}(B)$.
Then one has the following stability criterion {\cite{10}, Theorem 3}:

\begin{theorem}[Generalized Hoppe Criterion]
 Let $G\rightarrow X$ be a holomorphic vector bundle of rank $r\geq2$ over a polycyclic variety $X$ equiped with a polarisation 
 $\mathscr{L}$ if
 \[H^0(X,(\wedge^sG)\otimes\OO_X(B))=0\] 
 for all $B\in\pic(X)$ and $s\in\{1,\ldots,r-1\}$ such that
 $\begin{CD}\displaystyle{\delta_{\mathscr{L}}(B)<-s\mu_{\mathscr{L}}(G)}\end{CD}$ then $G$ is stable and if
 $\begin{CD}\displaystyle{\delta_{\mathscr{L}}(B)\leq-s\mu_{\mathscr{L}}(G)}\end{CD}$ then $G$ is semi-stable.\\
 Conversely if then $G$ is (semi-)stable then  \[H^0(X,G\otimes\OO_X(B))=0\]
 for all $B\in\pic(X)$ and all $s\in\{1,\ldots,r-1\}$ such that
 $\delta_{\mathscr{L}}(B)<-s\mu_{\mathscr{L}}(G)$ or $\delta_{\mathscr{L}}(B)\leq-s\mu_{\mathscr{L}}(G)$.
\end{theorem}

\vspace{0.75cm}

\noindent Suppose the ambient space is $X=\PP^{a_1}\times\PP^{a_1}\times\PP^{a_2}\times\PP^{a_2}\times\cdots\times\PP^{a_n}\times\PP^{a_n}$ then $\pic(X) \simeq \ZZ^{2n}$.\\
We denote by $g_1, g_2,\cdots,g_{2n}$ the generators of $\pic(X)$.\\
\\
Denote by $\OO_X(g_1, g_2,\cdots,g_{2n}):= {p_1}^*\OO_{\PP^{a_1}}(g_1)\otimes {p_2}^*\OO_{\PP^{a_1}}(g_2)\otimes {p_3}^*\OO_{\PP^{a_2}}(g_3)\otimes\cdots\otimes {p_{2n}}^*\OO_{\PP^{a_n}}(g_{2n})$,
where $p_1$ and $p_2$ are natural projections  from $X$ onto $\PP^{a_1}$, $p_3$ and $p_4$ are natural projections from $X$ onto $\PP^{a_2}$, $\ldots$ and $p_{2n-1}$ and $p_{2n}$ are natural projections  from $X$ onto $\PP^{a_n}$\\
\\
For any line bundle $\mathscr{L} = \OO_X(g_1, g_2,\cdots,g_{2n})$ on $X$ and a vector bundle $E$, we write 
$E(g_1, g_2,\cdots,g_{2n}) = E\otimes\OO_X(g_1, g_2,\cdots,g_{2n})$ 
and $(g_1, g_2,\cdots,g_{2n}):= 1\cdot[g_1\times\PP^{a_1}]+1\cdot[\PP^{a_1}\times g_2]+1\cdot[g_3\times\PP^{a_2}]+1\cdot[\PP^{a_2}\times g_4]+\cdots+1\cdot[g_{2n-1}\times \PP^{a_n}]+1\cdot[\PP^{a_n}\times g_{2n}]$ to represent its corresponding divisor.\\
\\
The normalization of $E$ on $X$ with respect to $\mathscr{L}$ is defined as follows:\\
Set $d=\deg_{\mathscr{L}}(\OO_X(1,0,\cdots,0))$, since $\deg_{\mathscr{L}}(E(-k_E,0,\cdots,0))=\deg_{\mathscr{L}}(E)-2nk\cdot \rk(E)$ 
there's a unique integer $k_E:=\lceil\mu_\mathscr{L}(E)/d\rceil$ such that $1 - d.\rk(E)\leq \deg_\mathscr{L}(E(-k_E,0,\cdots,0))\leq0$. 
The twisted bundle $E_{{\mathscr{L}}-norm}:= E(-k_E,0,\cdots,0)$ is called the $\mathscr{L}$-normalization of $E$.
Finally we define the linear functional $\delta_{\mathscr{L}}$ on  $\mathbb{Z}^{2n}$ as $\delta_{\mathscr{L}}(p_1,p_2,\cdots,p_{2n}):= \deg_{\mathscr{L}}\OO_{X}(p_1,p_2,\cdots,p_{2n})$.

\begin{proposition}
Let $X$ be a polycyclic variety with Picard number $2n$, let $\mathscr{L}$ be an ample line bundle and
let E be a rank $r>1 $ holomorphic vector bundle over $X$.
If $H^0(X,(\bigwedge^q E)_{{\mathscr{L}}-norm}(p_1,\cdots,p_{2n})) = 0$ for $1\leq q \leq r-1$ and every $(p_1,\cdots,p_{2n})\in \mathbb{Z}^{2n}$ such that $\delta_{\mathscr{L}}\leq0$
then E is $\mathscr{L}$-stable.
\end{proposition}

\begin{definition}
A vector bundle $E$ on $X$ is said to be
\begin{enumerate}
\renewcommand{\theenumi}{\alph{enumi}}
 \item indecomposable if it does not admit a direct sum decomposition of two proper subbundles $E_1$ and $E_2$ i.e.
 $E\ncong E_1\oplus E_2$ otherwise $E$ is decomposable.
\item simple if its only endomorphisms are the homotheties i.e. $\Hom(E,E)=k$ equivalently $h^0(X,E\otimes E^*)=1$.
\end{enumerate}
\end{definition}

\begin{proposition}
Let $0\rightarrow E \rightarrow F \rightarrow G\rightarrow0$ be an exact sequence of vector bundles. \\
Then we have the following exact sequences involving exterior and symmetric powers:\\
\begin{enumerate}
\renewcommand{\theenumi}{\alph{enumi}}
 \item $0\lra\bigwedge^q E \lra\bigwedge^q F \lra\bigwedge^{q-1} F\otimes G\lra\cdots \lra F\otimes S^{q-1}G \lra S^{q}G\lra0$\\
 \item $0\lra S^{q}E \lra S^{q-1}E\otimes F \lra\cdots \lra E\otimes\bigwedge^{q-1}F\lra\bigwedge^q F \lra\bigwedge^q G\lra 0$
\end{enumerate}
\end{proposition}

\begin{theorem}[K\"{u}nneth formula]
 Let $X$ and $Y$ be projective varieties over a field $k$. 
 Let $\mathscr{F}$ and $\mathscr{G}$ be coherent sheaves on $X$ and $Y$ respectively.
 Let $\mathscr{F}\boxtimes\mathscr{G}$ denote $p_1^*(\mathscr{F})\otimes p_2^*(\mathscr{G})$\\
 then $\displaystyle{H^m(X\times Y,\mathscr{F}\boxtimes\mathscr{G}) \cong \bigoplus_{p+q=m} H^p(X,\mathscr{F})\otimes H^q(Y,\mathscr{G})}$.
\end{theorem}

\begin{lemma}
 Let $X =\PP^{a_1}\times\PP^{a_1}\times\PP^{a_2}\times\PP^{a_2}\times\cdots\times\PP^{a_n}\times\PP^{a_n}$ then\\
$\displaystyle{H^t(X,\OO_X(p_1, p_2,\cdots,p_{2n}) )\cong \bigoplus_{q_1+q_2+\cdots+q_{2n}=t} Y_1\otimes Y_2\otimes\cdots\otimes Y_n}$  where\\
  $Y_1 = H^{q_1}(\PP^{a_1},\OO_{\PP^{a_1}}(p_1))\otimes H^{q_2}(\PP^{a_1},\OO_{\PP^{a_1}}(p_2))$, \\
  $Y_2 = H^{q_3}(\PP^{a_2},\OO_{\PP^{a_2}}(p_3))\otimes H^{q_4}(\PP^{a_2},\OO_{\PP^{a_2}}(p_4))$, $\cdots$ and \\
  $Y_n = H^{q_{2n-1}}(\PP^{a_n},\OO_{\PP^{a_n}}(p_{2n-1}))\otimes H^{q_{2n}}(\PP^{a_n},\OO_{\PP^{a_n}}(p_{2n}))$.
 \end{lemma}

\begin{theorem}[\cite{16}, Theorem 4.1]
 Let $n\geq1$ be an integer  and $d$ be an integer. We denote by $S_d$ the space of homogeneous polynomials of degree $d$ in 
 $n+1$ variables (conventionally if $d<0$ then $S_d=0$). Then the following statements are true:
 \begin{enumerate}
 \renewcommand{\theenumi}{\alph{enumi}}
  \item $H^0(\PP^n,\OO_{\PP^n}(d))=S_d$ for all $d$.
  \item $H^i(\PP^n,\OO_{\PP^n}(d))=0$ for $1<i<n$ and for all $d$.
  \item $H^n(\PP^n,\OO_{\PP^n}(d))\cong H^0(\PP^n,\OO_{\PP^n}(-d-n-1))$.
 \end{enumerate}
\end{theorem}

We generalize to a Cartesian $2n$-space lemma 2.13 \cite{13} for our purpose in this work.
\begin{lemma}
If  $p_1+p_2+\cdots+p_{2n}>0$ then $h^p(X,\OO_X (-p_1,-p_2,\cdots,-p_{2n})^{\oplus k}) = 0$ where $X = \PP^{a_1}\times\PP^{a_1}\times\PP^{a_2}\times\PP^{a_2}\times\cdots\times\PP^{a_n}\times\PP^{a_n}$ and for $0\leq p< \dim(X) -1$, for $k$ a nonnegative integer.
\end{lemma}

\begin{lemma}[\cite{11}, Lemma 10]
Let $A$ and $B$ be vector bundles canonically pulled back from $A'$ on $\PP^n$ and $B'$ on $\PP^m$ then\\
$\displaystyle{H^q(\bigwedge^s(A\otimes B))=
\sum_{k_1+\cdots+k_s=q}\big\{\bigoplus_{i=1}^{s}(\sum_{j=0}^s\sum_{m=0}^{k_i}H^m(\wedge^j(A))\otimes(H^{k_i-m}(\wedge^{s-j}(B)))) \big\}}$.
\end{lemma}

The proof of the lemma depends on the following:

\begin{enumerate}
 \renewcommand{\theenumi}{\alph{enumi}}
  \item $\displaystyle{H^q(A_1\oplus\cdots\oplus A_s) = \sum_{k_1+\cdots+k_s=q}\big\{\bigoplus_{i=1}^{s}H^k_i(A_i)\big\}}$.\\
  \item $\displaystyle{H^q(A\otimes B)=\sum_{m=0}^qH^m(A)\otimes H^{q-m}(B)}$.\\
  \item $\displaystyle{\wedge^s(A\otimes B)=\sum_{j=0}^s\wedge^j(A)\otimes\wedge^{s-j}(B)}$.
\end{enumerate}

For the sake of brevity we shall use the notation $H^q(\mathscr{F})$ in place of $H^q(X,\mathscr{F})$.

\section{Main Results}

\noindent The goal of this section is to construct monads over the multiprojective spaces.
We first establish the existence of monads on  $\PP^{a_1}\times\PP^{a_1}\times\PP^{a_2}\times\PP^{a_2}\times\cdots\times\PP^{a_n}\times\PP^{a_n}$, we then proceed  to prove stability and simplicity of the cohomology bundle $E$ associated to these monads on $X$.
We start by recalling the existence and classification of linear monads on $\PP^n$ given by
Fl\o{}ystad in \cite{4}.

\begin{lemma}[\cite{4}, Main Theorem] Let $k\geq1$. There exists monads on $\PP^k$ whose maps are matrices of linear forms,
\[
\begin{CD}
0@>>>{\OO_{\PP^{k}}(-1)^{\oplus a}} @>>^{A}>{\OO^{\oplus b}_{\PP^{k}}} @>>^{B}>{\OO_{\PP^{k}}(1)^{\oplus c}} @>>>0\\
\end{CD}
\]
if and only if at least one of the following is fulfilled;\\
$(1)b\geq2c+k-1$ , $b\geq a+c$ and \\
$(2)b\geq a+c+k$
\end{lemma}

\begin{theorem}
Let $X = \PP^1\times\PP^1\cdots\times\PP^1$ and $\mathscr{L} = \OO_X(1,\cdots,1)$ an ample line bundle. Denote by $N = h^0(\OO_X(1,\cdots,1)) - 1=2n+1$.
Then there exists a linear monad $M_\bullet$ on $X$ of the form
\[
\begin{CD}
M_\bullet: 0@>>>\OO_{X}(-1,\cdots,-1)^{\oplus\alpha}@>>^{f}>{\OO^{\oplus\beta}_X} @>>^{g}>\OO_{X}(1,\cdots,1)^{\oplus\gamma}  @>>>0\\
\end{CD}
\]
if and only if atleast one of the following is satified
\begin{enumerate}
\renewcommand{\theenumi}{\alph{enumi}}
 \item $\beta\geq 2\gamma + N -1$, and $\beta\geq \alpha + \gamma$,
 \item $\beta\geq \alpha + \gamma + N$, where $\alpha,\beta, \gamma$ be positive integers. 
\end{enumerate}
\end{theorem}

\noindent For certain values of $\alpha,\beta$ and $\gamma$ in the above monad the cohomology bundle is simple.
later we shall study these and write it up!!!!!!!!!

\vspace{0.5cm}

\noindent We now set up for monads on the Cartesian product $X = \PP^{a_1}\times\PP^{a_1}\times\PP^{a_2}\times\PP^{a_2}\times\cdots\times\PP^{a_n}\times\PP^{a_n}$.

\begin{lemma}
Let $a_1,\cdots,a_n$, $\alpha_1,\alpha_2,\cdots,\alpha_n$ and $k$ be positive integers, given $2n$ matrices $f_1,f_2,\cdots,f_{2n}$ as shown;
\vspace{0.5cm}
\[ f_1 =\left[ \begin{array}{ccccc}
&y^{\alpha_1}_{a_1} \cdots y^{\alpha_1+a_1k}_{0} \\
\adots&\adots \\
y^{\alpha_1}_{a_1} \cdots y^{\alpha_1+a_1k}_{0} \end{array} \right]_{k\times {(a_1+k)}}\]

\[ f_2 =\left[ \begin{array}{ccccccc}
&x^{\alpha_1}_{a_1} \cdots  x^{\alpha_1+a_1k}_{0}\\
\adots&\adots \\
x^{\alpha_1}_{a_1} \cdots  x^{\alpha_1+a_1k}_{0} \end{array} \right]_{k\times {(a_1+k)}}\]

\[ f_3 =\left[ \begin{array}{ccccc}
&y^{\alpha_2}_{a_2} \cdots y^{\alpha_2+a_2k}_{0} \\
\adots&\adots \\
y^{\alpha_2}_{a_2} \cdots y^{\alpha_2+a_2k}_{0} \end{array} \right]_{k\times {(a_2+k)}}\]

\[ f_4 =\left[ \begin{array}{ccccccc}
&x^{\alpha_2}_{a_2} \cdots  x^{\alpha_2+a_2k}_{0}\\
\adots&\adots \\
x^{\alpha_2}_{a_2} \cdots  x^{\alpha_2+a_2k}_{0} \end{array} \right]_{k\times {(a_2+k)}}\]

\[\vdots\]

\[ f_{2n-1} =\left[ \begin{array}{ccccccc}
&y^{\alpha_n}_{a_n} \cdots y^{\alpha_n+a_nk}_{0} \\
\adots&\adots \\
y^{\alpha_n}_{a_n} \cdots y^{\alpha_n+a_nk}_{0} \end{array} \right]_{k\times {(a_n+k)}}\]

\[ f_{2n} =\left[ \begin{array}{ccccccc}
&x^{\alpha_n}_{a_n} \cdots  x^{\alpha_n+a_nk}_{0}\\
\adots&\adots \\
x^{\alpha_n}_{a_n} \cdots  x^{\alpha_n+a_nk}_{0} \end{array} \right]_{k\times {(a_n+k)}}\]

\vspace{0.5cm}

and $2n$ additional matrices $g_1,g_2,\cdots,g_{2n}$ as shown;
\[ g_1 =\left[\begin{array}{cccccc}
x^{\alpha_1+a_1k}_{0}\\
\vdots &\ddots
 & x^{\alpha_1+a_1k}_{0}\\
x^{\alpha_1}_{a_1} &\ddots &\vdots\\
&& x^{\alpha_1}_{a_1}
\end{array} \right]_{{(a_1+k)}\times k}\] 

\[ g_2 =\left[\begin{array}{cccccc}
y^{\alpha_1+a_1k}_{0}\\
\vdots &\ddots
 & y^{\alpha_1+a_1}_{0}\\
y^{\alpha_1}_{a_1} &\ddots &\vdots\\
&& y^{\alpha_1}_{a_1}
\end{array} \right]_{{(a_1+k)}\times k}\] 

\[ g_3 =\left[\begin{array}{cccccc}
x^{\alpha_2+a_2k}_{0}\\
\vdots &\ddots
 & x^{\alpha_2+a_2k}_{0}\\
x^{\alpha_2}_{a_2} &\ddots &\vdots\\
&& x^{\alpha_2}_{a_2}
\end{array} \right]_{{(a_2+k)}\times k}\] 

\[ g_4 =\left[\begin{array}{cccccc}
y^{\alpha_2+a_2k}_{0}\\
\vdots &\ddots
 & y^{\alpha_2+a_2}_{0}\\
y^{\alpha_2}_{a_2} &\ddots &\vdots\\
&& y^{\alpha_2}_{a_2}
\end{array} \right]_{{(a_2+k)}\times k}\] 

\[\vdots\]

\[ g_{2n-1} =\left[ \begin{array}{cccccc}
x^{\alpha_n+a_nk}_{0}\\
\vdots &\ddots
 & x^{\alpha_n+a_nk}_{0}\\
x^{\alpha_n}_{a_n} &\ddots &\vdots\\
&& x^{\alpha_n}_{a_n}
\end{array} \right]_{{(a_n+k)}\times k}\] 

and

\[ g_{2n} =\left[\begin{array}{cccccc}
y^{\alpha_n+a_nk}_{0}\\
\vdots &\ddots
 & y^{\alpha_n+a_n}_{0}\\
y^{\alpha_n}_{a_n} &\ddots &\vdots\\
&& y^{\alpha_n}_{a_n}
\end{array} \right]_{{(a_n+k)}\times k}\] 

\vspace{0.5cm}

for non-negative integers $a$ and $b$  we define two matrices $f$ and $g$ as follows\\
\[ f =\left[\begin{array}{cccc}
f_1 & -f_2 & \cdots & -f_{2n}
         \end{array} \right]\] and

\[ g =\left[\begin{array}{cc}
g_1 \\ g_2 \\\vdots \\ g_{2n}
         \end{array} \right].\]
Then we have:\\
\begin{enumerate}
 \renewcommand{\theenumi}{\alph{enumi}}
  \item $f\cdot g = 0$ and
  \item The matrices $f$ and $g$ have maximal rank
\end{enumerate}
\end{lemma}

\begin{proof}
\begin{enumerate}
 \renewcommand{\theenumi}{\alph{enumi}}
  \item Since $f_1\cdot g_1 = f_2\cdot g_2$, $f_3\cdot g_3 = f_4\cdot g_4$ ,$\cdots ,f_{2n-1}\cdot g_{2n-1}= f_{2n}\cdot g_{2n}$ then we have that
\[ f\cdot g =\left[ \begin{array}{cccc} f_1 & -f_2 & \cdots & -f_{2n} \end{array} \right]
 \left[ \begin{array}{cc} g_1 \\ g_2 \\ \vdots \\ g_{2n}\end{array} \right]\] 
\[=[f_1g_1 - f_2g_2 + f_3g_3 - f_4g_4 +\cdots +f_{2n-1}g_{2n-1} - f_{2n}g_{2n}]\]\\
which is the zero matrix.
\\
 \item Notice that the rank of the two matrices drops if and only if all $x^{\alpha_1}_{0},\cdots,x^{\alpha_1}_{a_1}$, $x^{\alpha_2}_{0},\cdots,x^{\alpha_2}_{a_2}$, $\cdots$, $x^{\alpha_n}_{0},\cdots,x^{\alpha_n}_{a_n}$ and 
$y^{\alpha_1}_{0},\cdots,y^{\alpha_1}_{a_1}$, $y^{\alpha_2}_{0},\cdots,y^{\alpha_2}_{a_2}$, $\cdots$, $y^{\alpha_n}_{0},\cdots,y^{\alpha_n}_{a_n}$, are zeros and this is not possible in a projective space. Hence maximal rank.
\end{enumerate}
\end{proof}

\noindent Using the matrices given in the above lemma we are going to construct a monad.

\begin{theorem}
Let $a_1,\cdots, a_n$ and $k$ be positive integers. Then there exists a linear monad on $X =\PP^{a_1}\times\PP^{a_1}\times\PP^{a_2}\times\PP^{a_2}\times\cdots\times\PP^{a_n}\times\PP^{a_n}$ of the form;
\[
\begin{CD}
0@>>>{\OO_X(-1,\cdots,-1)^{\oplus k}} @>>^{f}>{\mathscr{G}_1\oplus\cdots\oplus\mathscr{G}_n}@>>^{g}>\OO_X(1,\cdots,1)^{\oplus k} @>>>0
\end{CD}
\]
where 
\begin{align*}
\mathscr{G}_1:=\OO_X(-1,0,0,\cdots,0)^{\oplus a_1+\oplus k}\oplus\OO_X(0,-1,0,0,\cdots,0)^{\oplus a_1+\oplus k}\\
\mathscr{G}_2:=\OO_X(0,0,-1,\cdots,0)^{\oplus a_2+\oplus k}\oplus\OO_X(0,0,0,-1,\cdots,0)^{\oplus a_2+\oplus k}\\
\cdots\cdots\cdots\cdots\cdots\cdots\cdots\cdots\cdots\cdots\cdots\cdots\cdots\cdots\cdots\cdots\cdots\cdots\cdots\\
\mathscr{G}_n:=\OO_X(0,0,\cdots,0,-1,0)^{\oplus a_n+\oplus k}\oplus\OO_X(0,0,\cdots,0,-1)^{\oplus a_n+\oplus k}
\end{align*}
\end{theorem}

\begin{proof}
The maps $f$ and $g$ in the monad are the matrices given in Lemma 3.4.\\
Notice that\\
$f\in$ Hom$(\OO_X(-1,\cdots,-1)^{\oplus k},\mathscr{G}_1\oplus\cdots\mathscr{G}_n)$ and \\
$g\in$ Hom$(\mathscr{G}_1\oplus\cdots\oplus\mathscr{G}_n,\OO_X(1,1,\cdots,1)^{\oplus k})$. \\
Hence by the above lemma they define the desired monad.
\end{proof}

\begin{theorem}
Let $T$ be a vector bundle on $X =\PP^{a_1}\times\PP^{a_1}\times\PP^{a_2}\times\PP^{a_2}\times\cdots\times\PP^{a_n}\times\PP^{a_n}$ defined by the short exact sequence
\[\begin{CD}0@>>>T @>>>\mathscr{G}_1\oplus\cdots\oplus\mathscr{G}_n@>>^{g}>\OO_X(1,\cdots,1)^{\oplus k} @>>>0\end{CD}\]
where 
\begin{align*}
\mathscr{G}_1:=\OO_X(-1,0,0,\cdots,0)^{\oplus a_1+\oplus k}\oplus\OO_X(0,-1,0,0,\cdots,0)^{\oplus a_1+\oplus k}\\
\mathscr{G}_2:=\OO_X(0,0,-1,\cdots,0)^{\oplus a_2+\oplus k}\oplus\OO_X(0,0,0,-1,\cdots,0)^{\oplus a_2+\oplus k}\\
\cdots\cdots\cdots\cdots\cdots\cdots\cdots\cdots\cdots\cdots\cdots\cdots\cdots\cdots\cdots\cdots\cdots\cdots\cdots\\
\mathscr{G}_n:=\OO_X(0,0,\cdots,0,-1,0)^{\oplus a_n+\oplus k}\oplus\OO_X(0,0,\cdots,0,-1)^{\oplus a_n+\oplus k}
\end{align*}
then $T$ is stable for an ample line bundle $\mathscr{L} = \OO_X(1,\cdots,1)$
\end{theorem}

\begin{proof}
We need to show that $H^0(X,\bigwedge^q T(-p_1,\cdots,-p_{2n}))=0$ for all $p_1+p_2+\cdots+p_{2n}>0$ and $1\leq q\leq \rk(T)$.\\
\\
Consider the ample line bundle $\mathscr{L} = \OO_X(1,\cdots,1) = \OO(L)$. \\
Its class in 
$\pic(X)= \langle [g_1\times\PP^a_1],[\PP^{a_1}\times g_2],[g_3\times\PP^{a_2}],[\PP^{2}\times g_4],\cdots,[g_{2n-1}\times\PP^{a_n}],[\PP^{a_n}\times g_{2n}]\rangle$ corresponds to the class\\
$1\cdot[g_1\times\PP^{a_1}]+1\cdot[\PP^{a_1}\times g_2]+1\cdot[g_3\times\PP^{a_2}]+1\cdot[\PP^{a_2}\times g_4]+\cdots+1\cdot[g_{2n-1}\times \PP^{a_n}]+1\cdot[\PP^{a_n}\times g_{2n}]$ and
$g_1,g_{2}$ are hyperplanes in $\PP^{a_1}$, $g_3,g_{4}$ are hyperplanes in $\PP^{a_2},\cdots$ and $g_{2n-1},g_{2n}$ are hyperplanes in $\PP^{a_n}$ with the intersection product induced by
$g_1^{a_1} = g_2^{a_1} = g_3^{a_2} = g_4^{a_2}=\cdots=g_{2n-1}^{a_n}=g_{2n}^{a_n}=1$ and $g_1^{a_1+1}= g_2^{a_1+1} = g_3^{a_2+1} = g_4^{a_2+1}=\cdots=g_{2n}^{a_n+1}=0$\\ 
\\
Now from the display diagram of the monad we get \\ 
$c_1(T) = c_1(\mathscr{G}_1\oplus\cdots\oplus\mathscr{G}_n) - c_1(\OO_X(1,\cdots,1)^{\oplus k})\\
       = (a_1+k)[(-1,0,\cdots,0)+(0,-1,\cdot,0)]+\cdots+(a_n+k)[(0,\cdots,-1,0)+(0,\cdots,0,-1)] - k(1,\cdots,1) \\
       = (-a_1-2k,-a_1-2k,-a_2-2k,-a_2-2k,\cdots,-a_n-2k,-a_n-2k) $\\
Since $L^{2a_1+\cdots+2a_n}>0$ the degree of $T$ is $\deg_{\mathscr{L}}T = c_1(T)\cdot\mathscr{L}^{d-1}$\\
\begin{align*}
\begin{split}
=-(a_1+a_2+\cdots+a_n+2nk)([g_1\times\PP^{a_1}]+[\PP^{a_1}\times g_2]+\cdots+[g_{2n-1}\times\PP^{a_n}]+[\PP^{a_n}\times g_{2n}])\\
(1\cdot[g_1\times\PP^{a_1}]+1\cdot[\PP^{a_1}\times g_2]+1\cdot[g_3\times\PP^{a_2}]+1\cdot[\PP^{a_2}\times g_4]+\cdots+1\cdot[\PP^{a_n}\times g_{2n}])^{2(a_1+\cdots+a_n)-1}\\
\end{split}
\end{align*}
$=-(a_1+\cdots+a_n+2nk)L^{2(a_1+\cdots+a_n)}< 0$\\
\\
Since $\deg_{\mathscr{L}}T<0$, then $(\bigwedge^q T)_{\mathscr{L}-norm} = (\bigwedge^q T)$ and  it suffices by 
the generalized Hoppe Criterion (Proposition 2.8), to prove that $h^0(\bigwedge^q T(-p_1,-p_2,\cdots,-p_{2n})) = 0$
with $p_1+p_2+\cdots+p_{2n}\geq0$ and for all $1\leq q\leq rk(T)-1$.\\
\\
Next we twist the exact sequence 
\[\begin{CD}0@>>>T @>>>\mathscr{G}_1\oplus\cdots\oplus\mathscr{G}_n@>>^{g}>\OO_X(1,\cdots,1)^{\oplus k} @>>>0\end{CD}\]
by $\OO_X(-p_1,\cdots,-p_{2n})$ we get,
\[
0\lra T(-p_1,\cdots,-p_{2n})\lra\mathscr{\overline{G}}_1\oplus\cdots\oplus\mathscr{\overline{G}}_n\lra\OO_X(1-p_1,\cdots,1-p_{2n})^{\oplus k}\lra0\]
where 
\begin{align*}
\mathscr{\overline{G}}_1:=\OO_X(-1-p_1,-p_2,-p_3\cdots,-p_{2n})^{\oplus a_1+\oplus k}\oplus\OO_X(-p_1,-1-p_2,-p_3,-p_4,\cdots,-p_{2n})^{\oplus a_1+\oplus k}\\
\mathscr{\overline{G}}_2:=\OO_X(-p_1,-p_2,-1-p_3,\cdots,-p_{2n})^{\oplus a_2+\oplus k}\oplus\OO_X(-p_1,-p_2,-p_3,-1-p_4,\cdots,-p_{2n})^{\oplus a_2+\oplus k}\\
\cdots\cdots\cdots\cdots\cdots\cdots\cdots\cdots\cdots\cdots\cdots\cdots\cdots\cdots\cdots\cdots\cdots\cdots\cdots\cdots\cdots\cdots\cdots\cdots\cdots\cdots\cdots\cdots\\
\mathscr{\overline{G}}_n:=\OO_X(-p_1,-p_2,\cdots,-1-p_{2n-1},-p_{2n})^{\oplus a_n+\oplus k}\oplus\OO_X(-p_1,-p_2\cdots,-p_{2n-1},-1-p_{2n})^{\oplus a_n+\oplus k}\\
\end{align*}

and taking the exterior powers of the sequence by Proposition 2.10 we get
\[0\lra \bigwedge^q T(-p_1,-p_2,\cdots,-p_{2n}) \lra \bigwedge^q (\mathscr{\overline{G}}_1\oplus\cdots\oplus\mathscr{\overline{G}}_n)\lra \cdots\]
\\
Taking cohomology we have the injection:
\[0\lra H^0(X,\bigwedge^{q}T(-p_1,-p_2,\cdots,-p_{2n}))\hookrightarrow H^0(X,\bigwedge^q (\mathscr{\overline{G}}_1\oplus\cdots\oplus\mathscr{\overline{G}}_n)\]
From Lemma 2.13 and 2.14 we have $H^0(X,\bigwedge^q (\mathscr{\overline{G}}_1\oplus\cdots\oplus\mathscr{\overline{G}}_n)$.\\
\\
$\Longrightarrow$ $h^0(X,\bigwedge^{q}T(-p_1,-p_2,\cdots,-p_{2n})) =  h^0(X,\bigwedge^q (\mathscr{\overline{G}}_1\oplus\cdots\oplus\mathscr{\overline{G}}_n)=0$\\
\\
i.e. $h^0(X,\bigwedge^{q}T(-p_1,-p_2,\cdots,-p_{2n}))=0$ and thus $T$ is stable.

\end{proof}

\begin{theorem} Let $X =\PP^{a_1}\times\PP^{a_1}\times\PP^{a_2}\times\PP^{a_2}\times\cdots\times\PP^{a_n}\times\PP^{a_n}$, then the cohomology vector bundle $E$ associated to the monad 
\[
\begin{CD}
0@>>>{\OO_X(-1,\cdots,-1)^{\oplus k}} @>>^{f}>{\mathscr{G}_1\oplus\cdots\oplus\mathscr{G}_n}@>>^{g}>\OO_X(1,\cdots,1)^{\oplus k} @>>>0
\end{CD}
\]
of rank $2(a_1+a_2+\cdots+a_n)+2k(n-1)$ is simple.
\end{theorem}

\begin{proof}
The display of the monad is
\[
\begin{CD}
@.@.0@.0\\
@.@.@VVV@VVV\\
0@>>>{\OO_{X}(-1,\cdots,-1)^{\oplus k}} @>>>T=\ker g@>>>E@>>>0\\
@.||@.@VVV@VVV\\
0@>>>{\OO_{X}(-1,\cdots,-1)^{\oplus k}} @>>^{f}>{\mathscr{G}_n\oplus\cdots\oplus\mathscr{G}_m}@>>>Q=\cok f@>>>0\\
@.@.@V^{g}VV@VVV\\
@.@.{\OO_{X}(1,\cdots,1)^{\oplus k}}@={\OO_{X}(1,\cdots,1)^{\oplus k} }\\
@.@.@VVV@VVV\\
@.@.0@.0
\end{CD}
\]

\noindent Since $T$ the kernel of the map $g$ is stable from the above theorem 3.6, we prove that the cohomology vector bundle $E=\ker g/\im f$  is simple.\\
\\
The first step is to take the dual of the short exact sequence 
\[\begin{CD}
0@>>>\OO_X(-1,\cdots,-1)^{\oplus k} @>>>T@>>>E @>>>0
\end{CD}\]
to get
\[
\begin{CD}
0@>>>E^* @>>>T^* @>>>\OO_X(1,\cdots,1)^{\oplus k}@>>>0.
\end{CD}
\]
Tensoring by $E$ we get\\
\[
\begin{CD}
0@>>>E\otimes E^* @>>>E\otimes T^* @>>>E(1,\cdots,1)^k@>>>0.
\end{CD}
\]
Now taking cohomology gives:
\[\begin{CD}
0@>>>H^0(X,E\otimes E^*) @>>>H^0(X,E\otimes T^*) @>>>H^0(E(1,\cdots,1)^{\oplus k})@>>>\cdots
\end{CD}\]
\\
which implies that 
\begin{equation}
h^0(X,E\otimes E^*) \leq h^0(X,E\otimes T^*)
\end{equation}
\\
Now we dualize the short exact sequence
\[\begin{CD}
0@>>>T @>>>{\mathscr{G}_1\oplus\cdots\oplus\mathscr{G}_n} @>>>\OO_X(1,\cdots,1)^{\oplus k} @>>>0
\end{CD}\]
\\
to get
\[\begin{CD}
0@>>>\OO_X(-1,\cdots,-1)^{\oplus k} @>>>{\mathscr{G}_1\oplus\cdots\oplus\mathscr{G}_n} @>>>T^* @>>>0
\end{CD}\]
\\
Now twisting the short exact sequence above by $\OO_X(-1,\cdots,-1)$ one obtains the short exact sequence
\[\begin{CD}
0@>>>\OO_X(-2,\cdots,-2)^{\oplus k} @>>>{\mathscr{G'}_1\oplus\cdots\oplus\mathscr{G'}_n} @>>>T^*(-1,\cdots,-1) @>>>0
\end{CD}\]
where 
\begin{align*}
\mathscr{G'}_1:=\OO_X(-2,-1,-1,\cdots,-1)^{\oplus a_1+\oplus k}\oplus\OO_X(-1,-2,-1,-1,\cdots,-1)^{\oplus a_1+\oplus k}\\
\mathscr{G'}_2:=\OO_X(-1,-1,-2,\cdots,-1)^{\oplus a_2+\oplus k}\oplus\OO_X(-1,-1,-1,-2,\cdots,-1)^{\oplus a_2+\oplus k}\\
\cdots\cdots\cdots\cdots\cdots\cdots\cdots\cdots\cdots\cdots\cdots\cdots\cdots\cdots\cdots\cdots\cdots\cdots\cdots\cdots\cdots\cdots\cdots\\
\mathscr{G'}_n:=\OO_X(-1,-1,\cdots,-1,-2,-1)^{\oplus a_n+\oplus k}\oplus\OO_X(-1,-1,\cdots,-1,-2)^{\oplus a_n+\oplus k}
\end{align*}
next on taking cohomology one gets\\
\[\begin{CD}
0\lra H^0(\OO_X(-2,\cdots,-2)^k) \lra H^0(\mathscr{G'}_1)\oplus\cdots\oplus H^0(\mathscr{G'}_n)\lra H^0(T^*(-1,\cdots,-1))\lra\\
\lra H^1(\OO_X(-2,\cdots,-2)^k) \lra H^1(\mathscr{G'}_1)\oplus\cdots\oplus H^1(\mathscr{G'}_n)\lra H^1(T^*(-1,\cdots,-1))\lra\\
\lra H^2(X,\OO_X(-2,\cdots,-2)^k) \lra H^2(\mathscr{G'}_1)\oplus\cdots\oplus H^2(\mathscr{G'}_n)\lra H^2(T^*(-1,\cdots,-1))\lra\cdots
\end{CD}
\]
\\
from which we deduce $H^0(X,T^*(-1,\cdots,-1)) = 0$ and $H^1(X,T^*(-1,\cdots,-1)) = 0$ from Theorems 2.11 and 2.12.\\
\\
Lastly, tensor the short exact sequence
\[
\begin{CD}
0@>>>\OO(-1,\cdots,-1)^{\oplus k} @>>>T @>>> E@>>>0\\
\end{CD}
\]
by $T^*$ to get
\[
\begin{CD}
0@>>>T^*(-1,\cdots,-1)^k @>>>T\otimes T^* @>>> E\otimes T^*@>>>0\\
\end{CD}
\]
and taking cohomology we have
\\
\[
\begin{CD}
0@>>>H^0(X,T^*(-1,\cdots,-1)^k) @>>>H^0(X,T\otimes T^*) @>>> H^0(X,E\otimes T^*)@>>>\\
@>>>H^1(X,T^*(-1,\cdots,-1)^k)@>>>\cdots
\end{CD}
\]
\\
But since  $H^0(X,T^*(-1,\cdots,-1)) = H^1(X,T^*(-1,\cdots,-1)) = 0$ from above then  it follows $H^1(X,T^*(-1,\cdots,-1)^k)=0$ for $k>1$.\\
\\
so we have 
\\
\[
\begin{CD}
0@>>>H^0(X,T^*(-1,\cdots,-1)^{k}) @>>>H^0(X,T\otimes T^*) @>>> H^0(X,E\otimes T^*)@>>>0
\end{CD}
\]
\\
This implies that 
\begin{equation}
h^0(X,T\otimes T^*) \leq h^0(X,E\otimes T^*)
\end{equation}
\\
Since $T$ is stable then it follows that it is simple which implies $h^0(X,T\otimes T^*)=1$.\\
\\
From $(3)$ and $(4)$ and putting these together we have;\\
\[1\leq h^0(X,E\otimes E^*) \leq h^0(X,E\otimes T^*) = h^0(X,T\otimes T^*) = 1\]\\
\\
We have $ h^0(X,E\otimes E^*) = 1 $ and therefore $E$ is simple.

\end{proof}

\section{Acknowledgment}
\noindent I wish to express sincere thanks to the Department of Mathematics, College of Science, Sultan Qaboos University staff for providing 
a conducive enviroment to be able to carry out research despite the overwhelming duties in teaching and community service.
I would also wish to express my sincere thanks to my collegues at the Department of Mathematics at the University of Nairobi for granting 
me leave in order to pursue my research work. Lastly, I am extremely grateful to my wife and our 3 kids who are always supportive of my pursuits.

\end{document}